\documentclass[12pt]{amsart}
\usepackage{latexsym}\usepackage{amsmath}\usepackage{amsthm}\usepackage{xspace}\usepackage{enumerate}\usepackage{amsfonts}\usepackage{amscd}\usepackage{syntonly}\usepackage{amssymb}
 
\usepackage{mathrsfs}
\usepackage{mathtools}
\usepackage[alwaysadjust]{paralist}
\usepackage{color}

\swapnumbers

\newtheorem{thm}{Theorem}
\newtheorem{prop}[thm]{Proposition}
\newtheorem{lemma}[thm]{Lemma}
\newtheorem{claim}{Claim}
\newtheorem*{claimnonumber}{Claim}

\newtheorem{rmk}[thm]{Remark}
\newtheorem*{rmknonumber}{Remark}

\newtheorem*{qnn}{Question}

\newenvironment{pf}[1][Proof.]{\noindent \emph{#1.}}{}

\newcommand\BAR[1]{\overline{#1}}

\newcommand\R{\mathbb{R}}
\newcommand\C{\mathbb{C}}
\newcommand\N{\mathbb{N}}

\newcommand\al{\alpha}

\newcommand\Om{\Omega}
\newcommand\om{\omega}
\newcommand\eps{\epsilon}
\renewcommand\phi{{\varphi}}

\newcommand\sub{\subseteq}
\newcommand\x{\times}
\newcommand\wo{\setminus}

\newcommand\coiso{{\operatorname{coiso}}}

\newcommand\Lag{{\operatorname{Lag}}}
\newcommand\D{\mathbb{D}}
\newcommand\id{\operatorname{id}}
\newcommand\wt[1]{\widetilde{#1}}
\newcommand\M{\mathcal{M}}

\newcommand\T{\mathbb{T}}
\newcommand\dd{\partial}
\newcommand\HH{\mathcal{H}}
\newcommand\RE{\operatorname{Re}}
\newcommand\IM{\operatorname{Im}}
\newcommand\Z{\mathbb{Z}}
\newcommand\im{\operatorname{im}}
\newcommand\LL{\mathcal{L}}

\newcommand\iinj{\hookrightarrow}

\title[Non-Squeezable Small Set, Regular Coisotropic Capacity]{A Symplectically Non-Squeezable Small Set and the Regular Coisotropic Capacity}

\author[Jan Swoboda]{Jan Swoboda}
\address{Max--Planck--Institut f\"ur Mathematik\\Vivatsgasse 7\\D-53111 Bonn\\Germany}
\email{swoboda@mpim-bonn.mpg.de}

\author[Fabian Ziltener]{Fabian Ziltener}
\address{Korea Institute for Advanced Study\\Hoegiro 87 (207-43 Cheongnyangni-dong)\\Dongdaemun-gu\\Seoul 130-722\\Republic of Korea}
\email{fabian@kias.re.kr}

\date{\today}

\begin{document}

\maketitle

\begin{abstract} 
We prove that for $n\geq2$ there exists a compact subset $X$  of the closed ball in $\R^{2n}$ of radius $\sqrt{2}$, such that $X$ has Hausdorff dimension $n$ and does not symplectically embed into the standard open symplectic cylinder. The second main result is a lower bound on the $d$-th regular coisotropic capacity, which is sharp up to a factor of $3$. For an open subset of a geometrically bounded, aspherical symplectic manifold, this capacity is a lower bound on its displacement energy. The proofs of the results involve a certain Lagrangian submanifold of linear space, which was considered by M.~Audin and L.~Polterovich.
\end{abstract}
\section{Motivation and results}\label{sec:main}
Continuing our previous work \cite{SZSmall,SZHofer}, the present article is motivated by the following question.
\begin{qnn}[A] How much symplectic geometry can a small subset of a symplectic manifold carry?
\end{qnn}
More concretely, we are concerned with the problem of finding a small subset of $\R^{2n}$ that cannot be squeezed symplectically. To be specific, we interpret ``smallness'' in two ways: in the sense of Hausdorff dimension and in terms of the size of a ball containing the subset. The first main result is the following. Let $(M,\om)$ and $(M',\om')$ be symplectic manifolds, and $X\sub M$ a subset. We say that \emph{$X$ (symplectically) embeds into $M'$} iff there exists an open neighborhood $U\sub M$ of $X$ and a symplectic embedding $\phi\colon U\to M'$. For $n\in\N$ and $a>0$ we denote by $B^{2n}(a)$ and $\BAR B^{2n}(a)$ the open and closed balls in $\R^{2n}$, of radius $\sqrt{a/\pi}$, around 0. (These balls have Gromov-width $a$.) We denote
\begin{eqnarray*}&B^{2n}:=B^{2n}(\pi),\quad\BAR B^{2n}\coloneqq\BAR B^{2n}(\pi),\quad\D\coloneqq\BAR B^2,&\\
&Z^{2n}(a)\coloneqq B^2(a)\x\R^{2n-2},\quad Z^{2n}\coloneqq Z^{2n}(\pi),&\\
&\BAR P_n:=\left\{\begin{array}{ll}
\D^n,&\textrm{if }n\textrm{ is even,}\\
\D^{n-1}\x\R^2,&\textrm{if }n\textrm{ is odd.}
\end{array}\right.&
\end{eqnarray*}
\begin{thm}[Non-squeezable small set]\label{thm:X} For every $n\geq2$ there exists a compact subset 
\[X\sub\BAR P_n\cap\BAR B^{2n}(2\pi)\]
of Hausdorff dimension $n$, which does not symplectically embed into the open cylinder $Z^{2n}$. In fact, we may choose this set to be the union of a closed\footnote{This means ``compact and without boundary''.} Lagrangian submanifold and the image of a smooth map from $S^2$ to $\R^{2n}$.\footnote{It follows from the hypothesis $n\geq2$ and standard arguments (cf.~\cite[p.~176]{Fe}) that such a union has Hausdorff dimension equal to $n$.}
\end{thm}
The set $X$ in this result is ``almost minimal'': If $n$ is even then the statement of Theorem \ref{thm:X} is wrong, if $\BAR P_n$ is replaced by $(\D\wo\{z\})\x\D^{n-1}$, where $z$ is an arbitrary point in $S^1=\dd\D$. This follows from an elementary argument, using compactness of $X$ and Moser isotopy in two dimensions. (A similar assertion holds in the case in which $n$ is odd.) Furthermore, the condition $X\sub\BAR B^{2n}(2\pi)$ is ``sharp up to a factor of 2''. In fact, based on a two-dimensional Moser type argument, we will show the following:
\begin{prop}\label{prop:2n-1} For $n\in\N$ every compact subset of $\BAR B^{2n}$ with vanishing $(2n-1)$-dimensional Hausdorff measure symplectically embeds into $Z^{2n}$. 
\end{prop}
In the proof of Theorem \ref{thm:X} we will consider a rotated and rescaled version $\wt L$ of a closed Lagrangian submanifold studied by L.~Polterovich in \cite{Polterovich}. We will choose a map from $S^2$ to $\R^{2n}$ with image equal to the union of the cones over some loops in $\wt L$ that generate the fundamental group of $\wt L$. The union $X$ of $\wt L$ and these cones cannot be squeezed into $Z^{2n}$. This will be a consequence of a result by Y.~Chekanov about the displacement energy of a Lagrangian submanifold.

We may ask whether the condition in Theorem \ref{thm:X} on the Hausdorff dimension of $X$ is optimal:
\begin{qnn}[B] Does every compact set $X\sub\R^{2n}$ of Hausdorff dimension $<n$ symplectically embed into an arbitrarily small symplectic cylinder or ball? Is this even true for any compact set $X$ with vanishing $n$-dimensional Hausdorff measure?
\end{qnn}
To our knowledge these questions are open.\\

Returning to Question (A), consider the class of ``small'' subsets of a given symplectic manifold consisting of coisotropic submanifolds. Based on these submanifolds, in \cite{SZSmall} for a fixed dimension $2n$ we defined a collection of capacities, one for each $d\in\{n,\ldots,2n-1\}$, as follows. Recall that a symplectic manifold $(M,\om)$ is called \emph{(symplectically) aspherical} iff for every $u\in C^\infty(S^2,M)$ we have $\int_{S^2}u^*\om=0$. For a coisotropic submanifold $N\sub M$ we denote by $A(N)=A(M,\om,N)$ its minimal (symplectic) area (or action). (See (\ref{eq:A N}) below.) We define the \emph{$d$-th regular coisotropic capacity} to be the map
\begin{eqnarray}\label{eq:A d coiso}&A^d_\coiso\colon\big\{\textrm{aspherical symplectic manifold, }\dim M=2n\big\}\to[0,\infty],&\\
\nonumber&A^d_\coiso(M,\om)\coloneqq\sup A(N),&
\end{eqnarray}
where $N\sub M$ runs over all non-empty closed regular (i.e., ``fibering'') coisotropic submanifolds of dimension $d$, satisfying the following condition:
\begin{equation}\label{eq:contractible}\forall \textrm{ isotropic leaf }F\textrm{ of }N,\,\forall x\in C(S^1,F)\colon\,x\textrm{ is contractible in }M.
\end{equation}
(For explanations see Subsection \ref{subsec:back}.) By \cite[Theorem 4]{SZSmall} the map $A^d_\coiso$ is a (not necessarily normalized) symplectic capacity. For $d=n$ we abbreviate
\[A_\Lag\coloneqq A^n_\coiso.\]
Since every Lagrangian submanifold is regular, $A_\Lag(M,\om)$ equals the supremum of all minimal areas $A(L)$, where $L$ runs over all those non-empty closed Lagrangian submanifolds of $M$, for which every continuous loop in $L$ is contractible in $M$. (Here $A(L)=\inf\big(S(L)\cap(0,\infty)\big)$, where the symplectic area spectrum $S(L)$ is given by (\ref{eq:S L}) below.)

Our second main result provides a lower bound on $A_\coiso^d$ for the unit ball $B^{2n}$, equipped with the standard symplectic form $\om_0$:
\begin{thm}[Regular coisotropic capacity]\label{thm:coiso} For every $n\geq2$ we have
\begin{eqnarray}
\label{thm:ineq1}&A_\Lag(B^{2n})\coloneqq A_\Lag(B^{2n},\om_0)\geq\frac\pi2,&\\
\label{thm:ineq2}&A^d_\coiso(B^{2n})\geq\frac\pi3,\quad\forall d\in\{n+1,\ldots,2n-3\}.&
\end{eqnarray}
\end{thm}
The proof of this result uses again the closed Lagrangian submanifold of $\R^{2n}$ studied by L.~Polterovich. To put Theorem \ref{thm:coiso} into context, note that in \cite[Theorem 4]{SZSmall} we proved the (in-)equalities
\begin{eqnarray*}&A^d_\coiso(Z^{2n})\leq\pi,\,\forall d\in\{n,\ldots,2n-1\},&\\
&A^{2n-1}_\coiso(B^{2n})=\pi,&\\
&A^{2n-2}_\coiso(B^{2n})\geq\frac\pi2.&
\end{eqnarray*}
Combining these with Theorem \ref{thm:coiso}, it follows that the capacity $A^d_\coiso$ is normalized for $d=2n-1$, normalized up to a factor of 2 for $d=n$ and $2n-2$, and up to a factor of 3, otherwise. 
\section{Remarks and related work}
\subsection*{About Theorem \ref{thm:X}}
Note that we may not just take a closed Lagrangian submanifold $L$ of $\R^{2n}$ for $X$, since every such submanifold ``symplectically embeds'' (in the above sense) into an arbitrarily small ball. To see this, let $B\sub\R^{2n}$ be an open ball. We choose a number $c>0$ such that the rescaled Lagrangian $cL$ is contained in $B$. It follows from Weinstein's neighborhood theorem that there exist open neighborhoods $U$ and $U'$ of $L$ and $cL$, respectively, and a symplectomorphism $\phi:U\to U'$ that maps $L$ to $cL$. The restriction of $\phi$ to $U\cap\phi^{-1}(B)$ is a symplectic embedding of a neighborhood of $L$ into $B$. 

Theorem \ref{thm:X} has the following application. For $n\in\N$ and $d\in[0,2n]$ consider the quantity
\[a(n,d)\coloneqq\inf a\in[0,\infty],\]
where the infimum runs over all numbers $a>0$, for which there exists a compact subset $X$ of $B^{2n}(a)$ of Hausdorff dimension at most $d$, such that $X$ does not symplectically embed into $Z^{2n}$. (Our convention is that $\inf\emptyset=\infty$.) Note that we always have $a(n,d)\geq\pi$, and $a(n,d)$ is decreasing in $d$. Theorem \ref{thm:X} implies that 
\[a(n,d)\leq2\pi,\quad\forall d\geq n,\]
and hence we know these numbers up to a factor of 2. This improves our previous result \cite[Theorem 6]{SZSmall}. That result implies that $a(n,d)$ is bounded above by $\pi$ times some integer, depending on $n$ and $d$ in a combinatorial way. For $n=d$ this integer behaves asymptotically like $\sqrt n$, as $n\to\infty$.

Gromov's non-squeezing result (cf.~\cite{Gr}) implies that $a(n,2n)=\pi$. This can be strengthened to the equality $a(n,2n-1)=\pi$, which follows from \cite[Theorem 6]{SZSmall}. In the case $d<2$ we have $a(n,d)=\infty$. This is a consequence of the following result. 
\begin{prop}[Two-dimensional squeezing]\label{prop:squeeze}  
For all $n\in\N$ and $a>0$, every subset $X$ of $\R^{2n}$ with vanishing $2$-dimensional Hausdorff measure symplectically embeds into $Z^{2n}(a)$.
\end{prop}
The proof of this result is based on Moser isotopy. In contrast with this proposition, a straight-forward argument shows that $a(1,2)=\pi$. Hence in the case $n=1$, the values $a(1,d)$ are all known.\\ 

Theorem \ref{thm:X} is related to the following results by J.-C.~Sikorav and F.~Schlenk. In \cite{Si} Sikorav proved that there does not exist a symplectomorphism of $\R^{2n}$ which maps $\T^n$ into $Z^{2n}$. Schlenk noted in \cite[p.~8]{SchlEmb} that combining this result with the Extension after Restriction Principle implies the ``Symplectic Hedgehog Theorem'': For every $n\geq2$, no starshaped domain in $\R^{2n}$ containing the torus $\T^n$ symplectically embeds into the cylinder $Z^{2n}$. It follows that no neighborhood of the set 
\[[0,1]\cdot\T^n:=\big\{cx\,\big|\,c\in[0,1],\,x\in\T^n\big\}\]
can be squeezed into $Z^{2n}$. This set has Hausdorff dimension $n+1$ and is contained in the ball $\BAR B^{2n}(n\pi)$. Theorem \ref{thm:X} improves this statement in two ways: The set $X$ in that result has Hausdorff dimension only $n$ and is contained in the ball $\BAR B^{2n}(2\pi)$, whose size does not depend on $n$. 
\subsection*{About Proposition \ref{prop:2n-1}}
In the case $n\geq2$ the condition on the Hausdorff measure in this result is necessary, since then no neighborhood of the unit sphere symplectically embeds into $Z^{2n}$. See \cite[Corollary 5]{SZSmall}.
\subsection*{About the regular coisotropic capacity and Theorem \ref{thm:coiso}} 
A motivation for the definition of $A_\coiso^d$ as in (\ref{eq:A d coiso}) is that for an open subset $U$ of an aspherical symplectic manifold $(M,\om)$ the number $A_\coiso^d(U)$ is a lower bound on the displacement energy of $U$, if $(M,\om)$ is geometrically bounded. This follows from \cite[Theorem 1.1]{ZiLeafwise}.

For $d=n$ the capacity $A_\Lag=A_\coiso^n$ is closely related to the Lagrangian capacity introduced by K.~Cieliebak and K.~Mohnke: We denote
\begin{align*}\M\coloneqq\big\{&(M,\om)\textrm{ symplectic manifold }\big|\\
&\dim M=2n,\,\pi_i(M)\textrm{ trivial },i=1,2\big\}.
\end{align*}
In \cite{CM}\footnote{See also \cite{CHLS}, Sec.~2.4, p.~11.} Cieliebak and Mohnke defined the \emph{Lagrangian capacity} to be the map $c_L\colon\M\to[0,\infty)$, given by
\[c_L(M,\om)\coloneqq\sup\big\{A(M,\om,L)\,\big|\,L\sub M\textrm{ embedded Lagrangian torus}\big\},\]
where $A(L)=\inf\big(S(L)\cap(0,\infty)\big)$ denotes the minimal symplectic area of $L$. The authors proved that 
\begin{equation}\label{eq:c L}c_L(B^{2n},\om_0)=\frac\pi n.\end{equation} 
The capacity $c_L$ is bounded above by $A_\Lag$. For $n\geq3$, it is strictly smaller than $A_\Lag$, when applied to $(B^{2n},\om_0)$. This follows from inequality (\ref{thm:ineq1}) and equality \eqref{eq:c L}.

For $d=2n-1$ the capacity $A_\coiso^{2n-1}$ is related to a definition recently introduced by H.~Geiges and K.~Zehmisch: In \cite{GZBall,GZWeinstein} these authors defined, for any symplectic manifold $(V,\om)$, 
\[c(V,\om):=\sup_{(M,\al)}\big\{\inf(\al)\,\big|\,\exists\textrm{ contact type embedding }(M,\al)\iinj(V,\om)\big\},\]
where the supremum is taken over all closed contact manifolds $(M,\al)$, and $\inf(\al)$ denotes the infimum of all positive periods of closed orbits of the Reeb vector field $R_\al$. They showed that $c$ is a normalized symplectic capacity. (See \cite[Theorem 4.5]{GZWeinstein}.) 

As a consequence of Theorem \ref{thm:coiso} and \cite[Theorem 4]{SZSmall}, the value of the capacity $A_\Lag=A_\coiso^n$ on the ball $B^{2n}$ lies between $\frac\pi2$ and $\pi$. In the case $n=2$ this value can be exactly calculated, if we modify the definition of $A_\Lag$ by restricting to \emph{orientable} Lagrangian submanifolds. Namely, the so obtained capacity $A_\Lag^+$ satisfies
\[A_\Lag^+(B^4)=\frac\pi2.\]
To see this, we denote by $\T^2=(S^1)^2$ the standard torus in $\R^4$. For every $r<\frac1{\sqrt2}$ the rescaled torus $r\T^2$ is a Lagrangian submanifold of $B^4$, with minimal area $\pi r^2$. It follows that $A_\Lag^+(B^4)\geq\frac\pi2$. To see the opposite inequality, note that every orientable closed connected Lagrangian submanifold $L\sub B^4$ is diffeomorphic to the torus $\T^2$, since its Euler characteristic vanishes. For such an $L$, K.~Cieliebak and K.~Mohnke proved \cite{CM} that $A(L)<\frac\pi2$. The statement follows.
\section{Background and proofs of the results of section \ref{sec:main}}\label{sec:back proofs}
\subsection{Background}\label{subsec:back}
Let $(M,\om)$ be a symplectic manifold and $N\sub M$ a submanifold. Then $N$ is called \emph{coisotropic} iff for every $x\in N$ the subspace
\[T_xN^\om=\big\{v\in T_xM\,\big|\,\om(v,w)=0,\,\forall w\in T_xN\big\}\]
of $T_xM$ is contained in $T_xN$. Examples include $N=M$, hypersurfaces, and Lagrangian submanifolds of $M$. Let $N\sub M$ be a coisotropic submanifold. Then $\om$ gives rise to the isotropic (or characteristic) foliation on $N$. We define the \emph{isotropy relation on $N$} to be the subset 
\[R^{N,\om}\coloneqq\big\{(x(0),x(1))\,\big|\,x\in C^\infty([0,1],N)\colon\dot x(t)\in(T_{x(t)}N)^\om,\,\forall t\big\}\]
of $N\x N$. This is an equivalence relation on $N$. For a point $x_0\in N$ we call the $R^{N,\om}$-equivalence class of $x_0$ the \emph{isotropic leaf} through $x_0$. (This is the leaf of the isotropic foliation that contains $x_0$.) We call $N$ \emph{regular} \label{regular} iff $R^{N,\om}$ is a closed subset and a submanifold of $N\times N$. This holds if and only if there exists a manifold structure on the set of isotropic leaves of $N$, such that the canonical projection $\pi_N$ from $N$ to the set of leaves is a submersion, cf.~\cite[Lemma 15]{ZiLeafwise}. If $N$ is closed then by C.~Ehresmann's theorem this implies that $\pi_N$ is a smooth (locally trivial) fiber bundle. (See the proposition on p.~31 in \cite{Eh}.) 

We define the \emph{(symplectic) area (or action) spectrum} and the \emph{minimal (symplectic) area} of $N$ as
\begin{eqnarray}\label{eq:S M om N}&S(N):=S(M,\om,N)\coloneqq&\\
\nonumber&\left\{\displaystyle\int_\D u^*\om\,\bigg|\,u\in C^\infty(\D,M)\colon\,\exists\textrm{ isotropic leaf }F\textrm{ of }N\colon\, u(S^1)\sub F\right\},&\\
\label{eq:A N}&A(N):=A(M,\om,N)\coloneqq\inf\big(S(M,\om,N)\cap (0,\infty)\big)\in[0,\infty].&
\end{eqnarray}
(Our convention is that $\inf\emptyset=\infty$.) Note that if $L=N$ is a Lagrangian submanifold of $M$ then the isotropic leaf of a point $x\in L$ is the connected component of $L$ containing $x$, and therefore the area spectrum of $L$ is given by
\begin{equation}\label{eq:S L}S(L)=\left\{\displaystyle\int_\D u^*\om\,\bigg|\,u\in C^\infty(\D,M)\colon\,u(S^1)\in L\right\}.\end{equation}
\subsection{Proof of Theorem \ref{thm:X} (Non-squeezable small set)}
The proof of Theorem \ref{thm:X} is based on the following result. 
\begin{prop}\label{prop:L} Let $n\geq2$ and $L\sub\R^{2n}$ be a non-empty closed Lagrangian submanifold. Then there exists a smooth map 
\[u:S^2\to[0,1]\cdot L\coloneqq\big\{cx\,\big|\,c\in[0,1],x\in L\big\}\sub\R^{2n},\]
such that the union $L\cup u(S^2)$ does not symplectically embed into the cylinder $Z^{2n}(A(\R^{2n},\om_0,L))$.
\end{prop}
The proof of Proposition \ref{prop:L} follows the lines of the proof of \cite[Proposition 21]{SZSmall}. It is based on the following result, which is due to Y.~Chekanov. Let $(M,\om)$ be a symplectic manifold. We denote by $\HH(M,\om)$ the set of all functions $H\in C^\infty\big([0,1]\x M,\R\big)$ whose Hamiltonian time $t$ flow $\phi_H^t\colon M\to M$ exists and is surjective, for every $t\in[0,1]$.\footnote{The time $t$ flow of a time-dependent vector field on a manifold $M$ is always an injective smooth immersion on its domain of definition. (This follows for example from \cite[Theorem 17.15, p.~451, and Problem 17-15, p.~463]{Le}.) Hence if it is everywhere well-defined and surjective then it is a diffeomorphism of $M$. The second condition is not a consequence of the first one. As an example, consider $M:=(0,\infty)\x\R$, $\om:=\om_0$, $H(q,p):=p$, and $t>0$. The Hamiltonian time $t$ flow of $H$ is everywhere well-defined and given by $\phi_H^t(q,p)=(q+t,p)$. However, the map $\phi_H^t:M\to M$ is not surjective.}

We define the \emph{Hofer norm} 
\[\Vert\cdot\Vert\colon\HH(M,\om)\to[0,\infty],\quad\Vert H\Vert\coloneqq\int_0^1\big(\sup_MH^t-\inf_MH^t\big)dt,\]
and the \emph{displacement energy} of a subset $X\sub M$ to be 
\[e(X,M,\om)\coloneqq\inf\big\{\Vert H\Vert\,\big|\,H\in\HH(M,\om)\colon\phi_H^1(X)\cap X=\emptyset\big\}.\footnote{Alternatively, one can define a displacement energy, using only functions $H$ with compact support. However, it seems more natural to allow for all functions in $\HH(M,\om)$. For some remarks on this issue see \cite{SZHofer}.}\]
\begin{thm}\label{thm:e} Let $L\sub M$ be a closed Lagrangian submanifold. Assume that $(M,\om)$ is geometrically bounded (see \cite{Ch}). Then we have
\[e(L,M,\om)\geq A(M,\om,L).\]
\end{thm}
\begin{proof}[Proof of Theorem \ref{thm:e}] This follows from the Main Theorem in \cite{Ch} by an elementary argument.  
\end{proof}
For the proof of Proposition \ref{prop:L}, we also need the following.
\begin{lemma}\label{le:A N} Let $(M,\om)$ and $(M',\om')$ be symplectic manifolds of the same dimension, $N\sub M$ a coisotropic submanifold, and $\phi\colon M\to M'$ a symplectic embedding. Assume that $(M',\om')$ is aspherical, and every continuous loop in a leaf of $N$ is contractible in $M$. Then we have
\[A\big(M',\om',\phi(N)\big)=A(M,\om,N).\]
\end{lemma}
\begin{proof}[Proof of Lemma \ref{le:A N}] This follows from \cite[Remark 32 and Lemma 33]{SZSmall}.
\end{proof}
\begin{proof}[Proof of Proposition \ref{prop:L}] Without loss of generality we may assume that $L$ is connected. We choose a point $x_0\in L$. Since $L$ is a closed manifold, there exists a finite set $\LL$ of loops in $L$ that generate the fundamental group $\pi_1(L,x_0)$. We choose these loops to be smooth, and define
\begin{eqnarray*}
&f:\LL\x[0,1]\x S^1\to\R^{2n},\quad f(x,t,z):=tx(z),&\\
&X:=L\cup\im(f).&
\end{eqnarray*}
The statement of the proposition is a consequence of the following two claims. 
\begin{claim}\label{claim:u} If $\LL\neq\emptyset$\footnote{By a result of M.~Gromov \cite{Gr} this is always the case. However, we do not use this in the proof of Proposition \ref{prop:L}.} then there exists a smooth map from $S^2$ to $\R^{2n}$ with the same image as $f$. 
\end{claim}
\begin{proof}[Proof of Claim \ref{claim:u}] We denote $k:=|\LL|$, and choose a bijection
\[\{1,\ldots,k\}\ni i\mapsto x_i\in\LL\] 
and a function $\rho\in C^\infty\big([0,1],[0,1]\big)$ that attains the value $i$ in a neighborhood of $i=0,1$. We define the map $u:[0,2k]\x S^1\to\R^{2n}$ by
\[u(t,z):=\left\{\begin{array}{ll}
\rho(t-2i+2)x_i(z),&\textrm{for }t\in[2i-2,2i-1],\\
\rho(2i-t)x_i(z),&\textrm{for }t\in[2i-1,2i],\\
\end{array}\right.\]
for $i=1,\ldots,k$. This map is smooth and has the same image as $f$. We identify $[0,2k]\x S^1$ with the two boundary circles collapsed with $S^2$. Since $u$ is constant in neighborhoods of $\{0\}\x S^1$ and $\{2k\}\x S^1$, it descends to a map from $S^2$ to $\R^{2n}$. This map has the required properties. This proves Claim \ref{claim:u}.
\end{proof}
\begin{claim}\label{claim:U} For every open neighborhood $U$ of $X$, and every symplectic embedding $\phi\colon U\to\R^{2n}$ we have $\phi(U)\not\sub Z^{2n}\big(A(\R^{2n},\om_0,L)\big)$. 
\end{claim}
\begin{pf}[Proof of Claim \ref{claim:U}] In order to apply Lemma \ref{le:A N}, we check that every continuous loop in $L$ is contractible in $U$. Let $x$ be such a loop. It follows from our choice of the set $\LL$ that there exist a collection of loops $y_1,\ldots,y_\ell\in \LL$ and signs $\eps_1,\ldots,\eps_\ell\in\{1,-1\}$, such that $x$ is homotopic inside $L$ to $y_1^{\eps_1}\#\cdots\#y_\ell^{\eps_\ell}$. Here $\#$ denotes concatenation of loops based at $x_0$, and $y_i^{-1}$ denotes the time-reversed loop $y_i$. Since $X$ contains the image of the map $[0,1]\x S^1\ni(t,z)\mapsto ty_i(z)\in\R^{2n}$, for every $i=1,\ldots,\ell$, it follows that $x$ is contractible in $X$, and hence in $U$. Therefore, the hypotheses of Lemma \ref{le:A N} are satisfied with $\big(M,\om,M',\om',N\big):=\big(U,\om_0|U,\R^{2n},\om_0,L\big)$. (Here $\om_0|U$ denotes the restriction of $\om_0$ to $U$.) Applying this result, it follows that 
\begin{equation}\label{eq:A U phi L}A\big(U,\om_0|U,L\big)=A\big(\R^{2n},\om_0,\phi(L)\big).
\end{equation}
Similarly, applying Lemma \ref{le:A N} with $\phi$ replaced by the inclusion map of $U$ into $\R^{2n}$, we have
\begin{equation}\label{eq:A U}A(\R^{2n},\om_0,L)=A\big(U,\om_0|U,L\big).
\end{equation}
By Theorem \ref{thm:e}, we have
\begin{equation}\label{eq:A e}A\big(\R^{2n},\om_0,\phi(L)\big)\leq e\big(\phi(L),\R^{2n},\om_0\big).
\end{equation}
An elementary argument shows that 
\[e\big(Z^{2n}(a),\R^{2n},\om_0\big)\leq a,\quad\forall a>0.\]
Combining this with (\ref{eq:A U phi L},\ref{eq:A U},\ref{eq:A e}), it follows that 
\begin{equation}\label{eq:A a}A(\R^{2n},\om_0,L)\leq a,\quad\forall a>0\textrm{ such that }\phi(L)\sub Z^{2n}(a).
\end{equation}
Assume by contradiction that $\phi(U)\sub Z^{2n}\big(A(\R^{2n},\om_0,L)\big)$. Since $L$ is compact and contained in $U$, it follows that $\phi(L)\sub Z^{2n}(a)$ for some number $a<A(\R^{2n},\om_0,L)$. This contradicts (\ref{eq:A a}). The statement of Claim \ref{claim:U} follows. This proves Proposition \ref{prop:L}.
\end{pf}
\end{proof}
In the proof of Theorem \ref{thm:X} we will apply Proposition \ref{prop:L} with a rotated and rescaled version of the Lagrangian submanifold
\begin{equation}\label{eq:L}L:=\big\{zq\,\big|\,z\in S^1\sub\C,\,q\in S^{n-1}\sub\R^n\big\}\subseteq\C^n.
\end{equation}
This submanifold was used by L.~Polterovich in \cite[Section 3]{Polterovich} as an example of a monotone Lagrangian in $\C^n$ with minimal Maslov number $n$. Previously, it was considered by A.~Weinstein in \cite[Lecture 3]{We} and M.~Audin in \cite[p.~620]{Audin}. 
\begin{lemma}\label{lem:minimalareaL}
For $n\geq2$ the minimal symplectic area of the Lagrangian $L$ in $\R^{2n}$ equals $\frac{\pi}{2}$.
\end{lemma}
\begin{proof}[Proof of Lemma \ref{lem:minimalareaL}] Let $n\geq2$. Recall the formula (\ref{eq:S L}) for the area spectrum $S(L)$. We write a point in $\R^{2n}$ as $(q,p)$, and denote by $\al:=q\cdot dp$ the Liouville one-form. Since $d\al=\om_0$, Stokes' theorem implies that 
\begin{equation}\label{eq:S L wt S L}S(L)=\wt S(L):=\left\{\int_{S^1}x^*\al\,\big|\,x\in C^\infty(S^1,L)\right\}.\end{equation}
To calculate $\wt S(L)$, we need the following.
\begin{claimnonumber} If $x:S^1\to L$, $\phi:[0,1]\to\R$, and $q:[0,1]\to S^{n-1}$ are smooth maps, such that
\begin{equation}\label{eq:x}x(e^{2\pi it})=e^{i\varphi(t)}q(t),\quad\forall t\in[0,1],\end{equation}
then we have
\begin{equation}\label{eq:x al}\int_{S^1}x^*\al=\frac{\phi(1)-\phi(0)}2.\end{equation}
\end{claimnonumber}
\begin{proof}[Proof of the claim] We have $|q|^2=1$ and $q\cdot\dot q=0$, and therefore,
\begin{eqnarray}
\nonumber\int_{S^1}x^*\al&=&\int_0^1\RE\big(e^{i\varphi}q\big)\cdot\IM\big(e^{i\varphi}(i\dot\varphi q+\dot q)\big)\,dt\\
\nonumber&=&\int_0^1\cos(\phi)^2\dot\phi\,dt\\
\label{eq:sin}&=&\left.\left(\frac14\sin(2\phi(t))+\frac{\phi(t)}2\right)\right|_{t=0}^1.
\end{eqnarray}
On the other hand, equality (\ref{eq:x}) implies that $\phi(1)-\phi(0)\in\pi\Z$, and therefore, the first term in (\ref{eq:sin}) vanishes. Equality (\ref{eq:x al}) follows. This proves the claim. 
\end{proof}
We show that $\wt S(L)\sub\frac\pi2\Z$: Let $x\in C^\infty(S^1,L)$. The map $\R\x S^{n-1}\ni(\phi,q)\mapsto e^{i\phi}q\in L\sub\C^n$ is a smooth covering map. Therefore, there exist smooth paths $\varphi\colon[0,1]\to\R$ and $q\colon[0,1]\to S^{n-1}$ such that equality (\ref{eq:x}) holds. It follows that $\phi(1)-\phi(0)\in\pi\Z$. Combining this with the claim, we obtain $\int_{S^1}x^*\al\in\frac\pi2\Z$. This shows that $\wt S(L)\sub\frac\pi2\Z$.

To prove the opposite inclusion, we choose a path $q\in C^\infty([0,1],S^{n-1})$ that is constant near the ends and satisfies $q(1)=-q(0)$. (Here we use that $n\geq2$, and therefore, $S^{n-1}$ is connected.) We define $x:S^1\to L$ by $x(e^{2\pi it}):=e^{\pi it}q(t)$, for $t\in[0,1)$. This is a smooth loop. By the above claim we have $\int_{S^1}x^*\al=\pi/2$. By considering multiple covers of $x$, it follows that $\wt S(L)\supseteq\frac\pi2\Z$.

Hence the equality $\wt S(L)=\frac\pi2\Z$ holds. Combining this with equality (\ref{eq:S L wt S L}), it follows that $A(L)=\pi/2$. This proves Lemma \ref{lem:minimalareaL}.
\end{proof} 
\begin{proof}[Proof of Theorem \ref{thm:X}] Let $n\geq2$. We define $L$ as in (\ref{eq:L}), and 
\begin{align*}&\wt L:=\\
&\big\{\sqrt2zw\,\big|\,z\in S^1\sub\C,\,w\in S^{2n-1}\sub\C^n:\,w_{n+1-j}=\BAR w_j,\,\forall j=1,\ldots,n\big\}.
\end{align*}
\begin{claimnonumber} There exists a unitary transformation $U$ of $\C^n$, such that $\wt L=\sqrt2UL$.
\end{claimnonumber}
\begin{proof}[Proof of the claim] The set 
\[W:=\big\{w\in\C^n\,\big|\,w_{n+1-j}=\BAR w_j,\,\forall j=1,\ldots,n\big\}\] 
is a Lagrangian subspace of $\C^n$. Therefore, there exists a unitary transformation $U$ of $\C^n$, such that $W=U\R^n$. The statement of the claim holds for every such $U$.
\end{proof}
We choose $U$ as in the claim. Since $U$ is a symplectic linear map, the set $\wt L$ is a Lagrangian submanifold of $\C^n$, and satisfies
\[A(\C^n,\om_0,\wt L)=2A(\C^n,\om_0,L).\]
By Lemma \ref{lem:minimalareaL} the right hand side equals $\pi$. Therefore, applying Proposition \ref{prop:L}, it follows that there exists a smooth map $u:S^2\to[0,1]\cdot\wt L$, such that the union $X:=\wt L\cup u(S^2)$ does not symplectically embed into the cylinder $Z^{2n}$. The set $X$ is contained in $\BAR B^{2n}(2\pi)$, since $\wt L\sub\BAR B^{2n}(2\pi)$.

Let $\wt w\in\wt L$. We choose $z\in S^1$ and $w\in S^{2n-1}$, such that $w_{n+1-j}=\BAR w_j$, for all $j$, and $\wt w=\sqrt2zw$. If $j\in\{1,\ldots,n\}$ is an index such that $j\neq\frac{n+1}2$, then we have
\[|\wt w_j|^2=2|w_j|^2=|w_j|^2+|w_{n+1-j}|^2\leq|w|^2=1.\]
Therefore, if $n$ is even then $\wt L$, and hence $X$ is contained in $\D^n$. It follows that $X$ has all the required properties in this case. Consider the case in which $n$ is odd. We denote $n=:2k+1$ and define 
\[\Psi:\C^n\to\C^n,\quad \Psi(w):=\big(w_1,\ldots,w_k,w_{k+2},\ldots,w_n,w_{k+1}\big).\]
It follows that $\Psi(\wt L)$ is contained in $\D^{n-1}\x\C$, and hence the same holds for $\Psi(X)$. Therefore, $\Psi(X)$ has the required properties. This proves Theorem \ref{thm:X}. 
\end{proof}
\subsection{Proof of Proposition \ref{prop:2n-1}}
The proof of this result is based on the following. Let $n\in\N$ and $U\sub\R^n$ be an open set. We denote by $|U|$ the volume of $U$. 
\begin{lemma}\label{le:volume} For every $c>|U|$ there exists an orientation and volume preserving embedding of $U$ into the open ball (around $0$) of volume $c$. 
\end{lemma}
The proof of this lemma is based on the following observation. For $r>0$ we denote by $B^n_r\sub\R^n$ the open ball (around $0$) of radius $r$.
\begin{rmk}\label{rmk:phi U} Let $U\sub\R^n$ be a non-empty open set, and $r>r_0>0$ real numbers. Then there exists an orientation preserving embedding $\phi$ of $U$ into the open ball in $\R^n$ of radius $r$, such that $B^n_{r_0}\sub\phi(U)$. This follows from an elementary argument.
\end{rmk}
\begin{proof}[Proof of Lemma \ref{le:volume}] By an elementary argument, we may assume without loss of generality that $U$ is connected and non-empty. It follows from Remark \ref{rmk:phi U} that there exists an orientation preserving embedding $\phi$ of $U$ into the open ball of volume $c$, such that the ball of volume $|U|$ is contained in $\phi(U)$. This condition ensures that $|\phi(U)|>|U|$. Hence composing $\phi$ with a shrinking homothety of $\R^n$, we obtain an orientation preserving embedding $\psi$ of $U$ into the ball of volume $c$, such that $|\psi(U)|=|U|$. Denoting by $\Om$ the standard volume form on $\R^n$, this means that $\int_U\Om=\int_U\psi^*\Om$. Therefore, a theorem by R.~Greene and K.~Shiohama (\cite[Theorem 1]{GS}) implies that there exists a diffeomorphism $\chi:U\to U$ such that $\chi^*\psi^*\Om=\Om$. (Here we use that $\int_U\Om<\infty$. The result is based on Moser isotopy.) The map $\psi\circ\chi$ has the required properties. This proves Lemma \ref{le:volume}.
\end{proof}
\begin{proof}[Proof of Proposition \ref{prop:2n-1}] Let $n\in\N$ and $X$ be a compact subset of $\BAR B^{2n}$ with vanishing $(2n-1)$-dimensional Hausdorff measure. Then $X$ does not contain $S^{2n-1}$, and hence there exists an orthogonal linear symplectic map $\Psi\colon\R^{2n}\to\R^{2n}$, such that $(1,0,\ldots,0)\not\in\Psi(X)$. Since $\Psi(X)$ is compact and contained in $\BAR B^{2n}$, an elementary argument shows that there exists $c<1$, such that
\begin{equation}\label{eq:phi X}\Psi(X)\sub Y:=\big\{(q,p)\in\D\mid q<c\big\}\x\R^{2n-2}.
\end{equation}
We choose an open neighborhood $U$ of $\big\{(q,p)\in\D\mid q<c\big\}$ of area less than $\pi$. By Lemma \ref{le:volume} $U$ symplectically embeds into the open unit ball in $\R^2$. Using \eqref{eq:phi X}, it follows that $\Psi(X)$ symplectically embeds into $Z^{2n}$. Hence the same holds for $X$. This proves Proposition \ref{prop:2n-1}.
\end{proof}
\subsection{Proof of Theorem \ref{thm:coiso} (Regular coisotropic capacity)}
The idea is to consider the Lagrangian submanifold $L$ defined in \eqref{eq:L} (for inequality (\ref{thm:ineq1})) and a product of it with a sphere (for inequality (\ref{thm:ineq2})). We need the following result. Recall the definition of the area spectrum (\ref{eq:S M om N}). 
\begin{lemma}\label{le:productactionspectrum}
Let $(M,\om)$ and $(M',\om')$ be symplectic manifolds, and $N\subseteq M$ and $N'\subseteq M'$ coisotropic submanifolds. Then 
\begin{eqnarray*}
S(M\times M',\omega\oplus\omega',N\times N')=S(M,\omega,N)+S(M',\omega',N').
\end{eqnarray*}
\end{lemma}
\begin{proof} We refer to \cite[Remark 31]{SZSmall}.
\end{proof}  
\begin{proof}[Proof of Theorem \ref{thm:coiso}] To prove {\bf inequality \eqref{thm:ineq1}}, we define $L$ as in \eqref{eq:L}. Let $r<1$. Then $rL$ is a closed Lagrangian submanifold of $B^{2n}$. Furthermore, condition (\ref{eq:contractible}) is satisfied with $(M,\om):=(B^{2n},\om_0)$, since $B^{2n}$ is contractible. An elementary argument using Lemmas \ref{lem:minimalareaL} and \ref{le:A N}, shows that $A(B^{2n},\om_0,rL)=\frac\pi2r^2$. Therefore, for every $r<1$ we have $A_\Lag(B^{2n},\om_0)\geq\frac\pi2r^2$. Inequality \eqref{thm:ineq1} follows. 

We prove {\bf inequality \eqref{thm:ineq2}}. Let $d\in\{n+1,\ldots,2n-3\}$. We define $L$ as in \eqref{eq:L} with $n$ replaced by $2n-d-1$. We denote by $S^{k-1}_r\sub\R^k$ the sphere of radius $r>0$, around 0. Let $r<1$. The set 
\begin{equation}\label{eq:N sqrt}N\coloneqq\sqrt{\frac23}rL\times S^{2d-2n+1}_{\sqrt{1/3}r}
\end{equation}
is a closed regular coisotropic submanifold of $B^{2n}$, of dimension $d$. Each factor has area spectrum in linear space given by $\frac{\pi r^2}{3}\mathbb Z$. (For the second factor this follows e.g.~from the proof of \cite[Proposition 1.3]{ZiLeafwise}.) Hence Lemma \ref{le:productactionspectrum} implies that $A(\R^{2n},\om_0,N)=\frac{\pi r^2}{3}$. Lemma \ref{le:A N} implies that this number equals $A(B^{2n},\om_0,N)$. It follows that $A_\coiso^d(B^{2n},\om_0)\geq\frac{\pi r^2}{3}$, for every $r<1$. Inequality \eqref{thm:ineq2} follows. This proves Theorem \ref{thm:coiso}. 
\end{proof}

\begin{rmknonumber} The ratio of the scaling factors used in the definition (\ref{eq:N sqrt}) above is optimal. Namely, for $r,r'>0$ consider the coisotropic submanifold $N\coloneqq rL\times S^{2d-2n+1}_{r'}$ of $\R^{2n}$. It follows from Lemma \ref{le:productactionspectrum} that 
\begin{equation}\label{eq:A gcd}A(\R^{2n},\om_0,N)=\pi\gcd\left\{\frac{r^2}2,{r'}^2\right\}.
\end{equation} 
Here we define the greatest common divisor of two real numbers $a,b$ to be 
\[\gcd\{a,b\}:=\sup\big\{c\in(0,\infty)\,\big|\,a,b\in c\Z\big\}.\]
(Here our convention is that the supremum over the empty set equals 0.) In order for $N$ to be contained in $B^{2n}$, we need $r^2+{r'}^2<1$. For a given $c<1$, the expression (\ref{eq:A gcd}) is largest (namely equal to $\frac{c\pi}3$) under the restriction $r^2+{r'}^2=c$, provided that $\frac{r^2}2={r'}^2$. This corresponds to the choice in (\ref{eq:N sqrt}).
\end{rmknonumber} 
\subsection{Proof of Proposition \ref{prop:squeeze} (Two-dimensional squeezing)} We denote by $Y\sub\R^2$ the image of $X$ under the canonical projection from $\R^{2n}=\R^2\x\R^{2n-2}$ onto the first component. The 2-dimensional Hausdorff measure of $Y$ vanishes by a standard result. (See e.g.~\cite[p.~176]{Fe}.) Therefore, there exists an open neighborhood $U\sub\R^2$ of $Y$ of area less than $a$. By Lemma \ref{le:volume} there exists a symplectic embedding $\phi$ of $U$ into the open ball in $\R^2$, of area $a$. The product $U\x\R^{2n-2}$ is an open neighborhood of $X$, and $\phi\x\id$ is a symplectic embedding of this neighborhood into $Z^{2n}(a)$. This proves Proposition \ref{prop:squeeze}. $\Box$
\subsection*{Acknowledgments}
We would like to thank Felix Schlenk for making us aware of the Lagrangian submanifold considered by L.~Polterovich in \cite{Polterovich}. A considerable part of the work on this article was carried out while the second author visited the Max Planck Institute for Mathematics in Bonn. He would like to express his gratitude to the institute for the invitation and the generous fellowship. Both authors would also like to thank the Hausdorff Institute for Mathematics in Bonn, which hosted them during its Junior Hausdorff Trimester Program in Differential Geometry. We are also grateful to the anonymous referee for his/ her valuable suggestions.

\begin{thebibliography}{99}

\bibitem[Au]{Audin}M.~Audin, \emph{Fibr\'es normaux d'immersions en dimension double, points doubles d'immersions lagrangiennes et plongements totalement r\'eels}, Comment.~Math.~Helv.~{\bf 63} (1988), 593--623. 

\bibitem[BT]{BT} R.~Bott and L.~W.~Tu, \emph{Differential forms in algebraic topology}, Graduate Texts in Mathematics, {\bf 82}, Springer-Verlag, New York--Berlin, 1982, xiv+331 pp. 

\bibitem[Ch]{Ch} Yu.~Chekanov, \emph{Lagrangian intersections, symplectic energy, and areas of holomorphic curves}, Duke Math.~J.~{\bf 95} (1998), no.~{\bf 1}, 213--226.

\bibitem[CHLS]{CHLS} K.~Cieliebak, H.~Hofer, J.~Latschev and F.~Schlenk, \emph{Quantitative symplectic geometry}, in \emph{Dynamics, ergodic theory, and geometry}, 1-44, Math.~Sci.~Res.~Inst.~Publ., 54, Cambridge University Press, Cambridge, 2007.

\bibitem[CM]{CM} K.~Cieliebak and K.~Mohnke, \emph{Punctured holomorphic curves and Lagrangian embeddings}, preprint 2003. 

\bibitem[Eh]{Eh} C.~Ehresmann, \emph{Les connexions infinit\'esimales dans un espace fibr\'e diff\'erentiable}, Colloque de topologie (espaces fibr\'es), Bruxelles, 1950, 29--55.

\bibitem[Fe]{Fe} H.~Federer, \emph{Geometric measure theory}. Die Grundlehren der mathematischen Wissenschaften, {\bf 153}, Springer-Verlag, New York, 1969.

\bibitem[GZ1]{GZBall} H.~Geiges and K.~Zehmisch, \emph{How to recognise a 4-ball when you see one}, preprint, arXiv:1104.1543.

\bibitem[GZ2]{GZWeinstein} H.~Geiges and K.~Zehmisch, \emph{Symplectic cobordisms and the strong Weinstein conjecture}, Math.~Proc.~Camb.~Phil.~Soc., 2012, 
doi:10.1017/S0305004112000163. 

\bibitem[GS]{GS} R.~Greene and K.~Shiohama, \emph{Diffeomorphisms and volume preserving embeddings of non-compact manifolds}, Trans.~AMS {\bf 225} (1979), 403--414.

\bibitem[Gr]{Gr} M.~Gromov, \emph{Pseudoholomorphic curves in symplectic manifolds}, Invent.~Math.~{\bf 82} (1985), no.~2, 307--347.

\bibitem[Ha]{Ha} A.~Hatcher, \emph{Notes on Basic 3-manifold topology}, http://www.math.cornell.edu/~hatcher/3M/3Mdownloads.html.

\bibitem[Le]{Le} J.~M.~Lee, \emph{Introduction to smooth manifolds}, Graduate Texts in Mathematics, {\bf 218}, Springer-Verlag, New York, 2003. 

\bibitem[Mo]{Mo} J.~Moser, \emph{On the volume elements on a manifold}, Trans.~Amer.~Math.~Soc.~{\bf 120} (1965), 286--294. 

\bibitem[Po]{Polterovich}L.~Polterovich, \emph{Monotone Lagrange submanifolds of linear spaces and the Maslov class in cotangent bundles}, Math.~Z.~{\bf 207} (1991), no.~2, 217--222.

\bibitem[Schl1]{SchlVol} F.~Schlenk, \emph{Volume preserving embeddings of open subsets of $\R^n$ into mani-folds}, Proc.~Amer.~Math.~Soc.~{\bf 131} (2003), {\bf no.~6}, 1925--1929 (electronic). 

\bibitem[Schl2]{SchlEmb} F.~Schlenk, \emph{Embedding problems in symplectic geometry}, de Gruyter Expositions in Mathematics {\bf 40}, W.~de Gruyter, 2005. 

\bibitem[Si]{Si} J.--C.~Sikorav, \emph{Quelques propri\'et\'es des plongements lagrangiens}, Analyse globale et physique math\'ematique (Lyon, 1989), M\'em.~Soc.~Math.~France {\bf 46} (1991), 151--167.

\bibitem[SZ1]{SZSmall} J.~Swoboda and F.~Ziltener, \emph{Coisotropic Displacement and Small Subsets of a Symplectic Manifold}, Math.~Z.~{\bf 271} (2012), {\bf no.~1}, 415--445.

\bibitem[SZ2]{SZHofer} J.~Swoboda and F.~Ziltener, \emph{Hofer Geometry of a Subset of a Symplectic Manifold}, arXiv:1102.4889, Online First, Geom.~Dedicata, to appear in printed form. 

\bibitem[We]{We} A.~Weinstein, \emph{Lectures on symplectic manifolds}, Regional Conference Series in Mathematics, {\bf no.~29}, American Mathematical Society, Providence, R.I., 1977.

\bibitem[Zi]{ZiLeafwise} F.~Ziltener, \emph{Coisotropic Submanifolds, Leafwise Fixed Points, and Presymplectic Embeddings}, J.~Symplectic Geom.~{\bf 8} (2010), no.~{\bf 1}, 1--24.
%
%
\end{thebibliography}
\end{document}